\documentclass[11pt,letterpaper]{amsart}
\usepackage[letterpaper,hmargin=1.25in,vmargin=1.5in]{geometry}
\usepackage{latexsym, amssymb, amsmath}

\newtheorem{theorem}{Theorem}[section]
\newtheorem{lemma}[theorem]{Lemma}
\newtheorem{corollary}[theorem]{Corollary}
\newtheorem{proposition}[theorem]{Proposition}
\newtheorem{remark}[theorem]{Remark}
\theoremstyle{definition}

\newtheorem*{example}{Example}

\newtheorem{claim}{Claim}

\makeatletter
    
    \@addtoreset{equation}{section}
  \makeatother

\newcommand{\Ric}{{\rm Ric}}

\begin{document}

\title[Classification of Expanding Gradient Yamabe Solitons]
{Triviality, Rotational Symmetry, and Classification of Complete Expanding Gradient Yamabe Solitons}

\author{Shun Maeta}
\address{Department of Mathematics, Chiba University, 1-33, Yayoicho, Inage, Chiba, 263-8522, Japan.}
\curraddr{}
\email{shun.maeta@faculty.gs.chiba-u.jp~{\em or}~shun.maeta@gmail.com}
\thanks{The author is partially supported by the Grant-in-Aid for Scientific Research (C), No.23K03107, Japan Society for the Promotion of Science.}
\subjclass[2010]{53C21, 53C25, 53C20}

\date{}

\dedicatory{}

\keywords{Expanding gradient Yamabe solitons; Yamabe flow; Scalar curvature}

\commby{}

\begin{abstract}
In this paper, we rigorously analyze the scalar curvature of complete expanding gradient Yamabe solitons. We completely classify nontrivial complete expanding gradient Yamabe solitons in both cases: when the scalar curvature is greater than the soliton constant and when it is less than the soliton constant.
\end{abstract}

\maketitle

\bibliographystyle{amsplain}


\section{Introduction}\label{intro}

The Yamabe flow is a fundamental and highly effective tool for analyzing the structure of manifolds.
The study of the Yamabe flow has developed rapidly in the last decades and has become a central area in geometry (cf. \cite{Brendle05}, \cite{Brendle07}, \cite{Chow92}, \cite{CLN06}, \cite{Hamilton89}, \cite{SS03}, \cite{Ye94}). 
To analyze the Yamabe flow, and more generally, geometric flows, it is important to study singularity models.
Yamabe solitons are special solutions to the Yamabe flow and are expected to serve as singularity models (cf. \cite{CLN06}).  From this perspective, Yamabe solitons have also been studied extensively over the last few decades  (cf. \cite{CSZ12}, \cite{CMM12}, \cite{DS13}, \cite{Maeta28}, \cite{Maeta29}).

An $n$-dimensional Riemannian manifold $(M^n,g)$ is called a gradient Yamabe soliton if there exists a smooth function $F$ on $M$ and a constant $\lambda\in \mathbb{R}$ such that 
\[
\nabla \nabla F=(R-\lambda)g,
\]
where $\nabla \nabla F$ is the Hessian of $F$ and $R$ denotes the scalar curvature of $(M,g)$.
A gradient Yamabe soliton is represented as $(M^n, g, F, \lambda)$.
If $F$ is constant, then the gradient Yamabe soliton is called trivial.
If $\lambda>0$, $\lambda=0$, or $\lambda<0$, then the Yamabe soliton is called shrinking, steady, or expanding.

It is well known that there exist no nontrivial compact gradient Yamabe solitons (cf. \cite{Hamilton89}). One of the main problems in the study of Yamabe solitons is determining whether complete gradient Yamabe solitons are rotationally symmetric. 
In particular, the Yamabe-soliton analogue of Perelman's conjecture: {\it Under what conditions do nontrivial complete gradient shrinking, steady, and expanding Yamabe solitons with positive (scalar) curvature or with bounded (scalar) curvature become rotationally symmetric? In particular, what natural geometric assumptions, beyond positive or bounded (scalar) curvature, are required?} Partial answers have been obtained under the assumption of local conformal flatness (cf. \cite{CSZ12}, \cite{DS13}).
Catino, Mantegazza and Mazzieri also provided a partial answer to the problem under the assumption that $\Ric(\nabla F,\nabla F)\geq0$ and $\Ric (\nabla F,\nabla F)>0$ holds at least at one point (cf. \cite{CMM12}). Here, $\Ric$ is the Ricci tensor of $M.$ 
C. He resolved the steady case (cf. \cite{He11}).
For the shrinking case, the author resolved the problem (cf. \cite{Maeta28}).
Therefore, the remaining problem involves expanding solitons.
However, the study of expanding solitons has been regarded as difficult so far. 
In fact, even in low-dimensional cases, expanding gradient Yamabe solitons exhibit richer geometry than steady or shrinking solitons (cf. \cite{BM15}, \cite{Maeta29}), making it difficult to obtain rigidity results.
To overcome this difficulty, we reduce the problem to the existence and nonexistence of solutions of a certain ODE. 
Cao, Sun, and Zhang showed the structure theorem for Yamabe solitons (see also \cite{Tashiro65} and \cite{CMM12}).

\begin{theorem}[\cite{Tashiro65}, \cite{CSZ12}, \cite{CMM12}]\label{CSZ}
Let $(M,g,F)$ be a nontrivial, complete, gradient Yamabe soliton. 
Then $(M,g)$ is a complete warped product manifold and must take one of the two forms:

\begin{enumerate}
\item[$(1)$]
a warped product manifold
$([0,+\infty),dr^2)\times_{{F'(r)}^2}(\mathbb{S}^{n-1},{\bar g}_{S})$,
where $\bar g_{S}$ is the round metric on $\mathbb{S}^{n-1},$ or
\item[$(2)$]
a warped product manifold
$(\mathbb{R},dr^2)\times_{{F'(r)}^2} \left(N^{n-1},\bar g\right),$
where the scalar curvature $\bar R$ of $N$ is constant. If the scalar curvature of $M$ is nonnegative, then $
\bar R>0$ or $R=\bar R=0.$ 
\end{enumerate}
\end{theorem}

\begin{remark}\label{depends r}
~
\begin{enumerate}

\item[$(1)$]
For a conformal soliton, that is, a Riemannian manifold $(M,g,F,\varphi)$ that satisfies the condition $\nabla\nabla F=\varphi g$, where $\varphi\in C^\infty(M)$, Tashiro \cite{Tashiro65} provided the structure theorem, and Catino, Mantegazza, and Mazzieri also provided the structure theorem in \cite{CMM12} $($see also \cite{Maeta21}$).$ 
Such manifolds were also studied by Cheeger and Colding \cite{CC96}.

\item[$(2)$] It is shown that $F$ depends only on $r$, and in Case (2) of Theorem \ref{CSZ}, without loss of generality, we can assume that $\rho(r) = F'(r) > 0$ on $\mathbb{R}$ (cf. \cite{CSZ12} see also \cite{Maeta21}).
\end{enumerate}
\end{remark}

By the proof of Theorem \ref{CSZ} (see also \cite{CMM12}, \cite{Maeta21}), it is easy to see that if $\Ric(\nabla F,\nabla F)\geq0$ and $\Ric (\nabla F,\nabla F)>0$ holds at least at one point, then any such complete Riemannian manifold $(M,g,F, \varphi)$ with a non-constant function $F$ is rotationally symmetric. Therefore, in the study of the rotational symmetry of complete expanding gradient Yamabe solitons, imposing conditions on the scalar curvature is of fundamental importance.
To analyze the scalar curvature, the following equations, derived from the structure of the warped product, are used.
\[
|\nabla F|^2R=\bar R-(n-1)(n-2)(R-\lambda)^2-2(n-1)g(\nabla F,\nabla R).
\]
Set $\rho=F'$. Then we have
\begin{equation}\label{key2}
2(n-1)\rho\rho''+(n-1)(n-2)\rho'^2+\rho^2(\rho'+\lambda)=\bar R.
\end{equation}
Hence, the scalar curvature $\bar R$ of $N$ is constant. Differentiating both sides of \eqref{key2}, 
\begin{equation}\label{key3}
2(n-1)^2\rho'\rho''+2(n-1)\rho\rho'''+2\rho\rho'(\rho'+\lambda)+\rho^2\rho''=0.
\end{equation}
Note that, since low dimensional complete gradient solitons have already been classified (cf. \cite{BM15, Maeta29}), we focus on the case $n = \dim M \ge 3$ in this paper.

\section{Examples and rotational symmetry}

To study rotational symmetry of complete expanding gradient Yamabe solitons, we first note that there exist nontrivial examples when $R=\lambda$ (cf. \cite{Maeta29}).

\begin{example}
Let $(N^{n-1},~\bar g)$ be an $(n-1)$-dimensional complete Riemannian manifold with constant negative scalar curvature $\bar R$. Then, for any $\alpha \in \mathbb{R}$, $(M,g,F,\lambda)=(\mathbb{R}\times N^{n-1}, dr^2+\frac{\bar R}{\lambda}\bar g, \sqrt{\frac{\bar R}{\lambda}} r+\alpha,\lambda)$ is an $n$-dimensional complete expanding gradient Yamabe soliton with $R=\lambda.$

In particular, $(M^3,g,F,\lambda)=(\mathbb{R}\times \mathbb{H}^2, dr^2+\frac{\bar R}{\lambda} {\bar g}_{H}, \sqrt{\frac{\bar R}{\lambda}} r+\alpha,\lambda)$ is a 3-dimensional complete expanding gradient Yamabe soliton with $R=\lambda$, where $(\mathbb{H}^2,{\bar g}_H)$ is a hyperbolic surface.

\end{example}

Therefore, imposing conditions on the scalar curvature is crucial and fundamental for achieving rotational symmetry in complete expanding gradient Yamabe solitons.

The following lemma plays a fundamental role in the proof of our results later.

\begin{lemma}[\cite{Maeta21}]\label{equivlambda}
Let $(M,g,F,\lambda)$ be a nontrivial, complete, gradient Yamabe soliton such that 
$M=\mathbb{R}\times N^{n-1}$ and $g=dr^2+{{F'(r)}^2} \bar g$.
If $R\geq\lambda$ (resp. $R\leq\lambda$) on $M$, then $R>\lambda$ or $R\equiv\lambda$ (resp. $R<\lambda$ or $R\equiv\lambda$) on $M$.
\end{lemma}

 In particular, the problem reduces to the cases $R>\lambda$ or $R<\lambda$.
If $R>\lambda+\varepsilon$ or $R<\lambda-\varepsilon$, then these manifolds are rotationally symmetric.

\begin{proposition}\label{R>le}
Let $(M,g,F,\lambda)$ be a nontrivial, complete, gradient Yamabe soliton. If $R>\lambda+\varepsilon$ or $R<\lambda-\varepsilon$ for some $\varepsilon>0$, then it is rotationally symmetric and is given by the warped product 
\[
([0,+\infty),dr^2)\times_{|\nabla F|}(\mathbb{S}^{n-1},{\bar g}_{S}).
\]
\end{proposition}

\begin{proof}
Set $\rho(r)=F'(r)$. We only need to consider (2) of Theorem \ref{CSZ}. 
If $R>\lambda+\varepsilon$, then we have $\rho'>\varepsilon$ and $\rho>0$, we have a contradiction. 
If $R<\lambda-\varepsilon$, then one has that the positive smooth function $\rho$ satisfies $\rho'<-\varepsilon$. Therefore, we have a contradiction.
 \end{proof}
 
As a corollary, we obtain rotational symmetry of complete expanding gradient Yamabe solitons with nonnegative scalar curvature.

\begin{corollary}
Any nontrivial, complete, expanding gradient Yamabe soliton with nonnegative scalar curvature is rotationally symmetric.
\end{corollary}

This generalizes Theorem 1.3 in \cite{DS13}, Corollary 3.3 in \cite{CMM12} and Corollary 1.6 in \cite{CSZ12}. In particular, we remove the assumption of local conformal flatness.

\section{Case $R\leq\lambda$}\label{sectionR<l}

We consider the case $R\leq\lambda$ and completely classify complete expanding gradient Yamabe solitons.

\begin{theorem}\label{R<l}
Let $(M^n,g,F,\lambda)$ be a nontrivial, complete, expanding gradient Yamabe soliton with $R\leq\lambda$. Then, $(M^n,g,F,\lambda)$ is either 
\begin{enumerate}
\item[$(1)$]
 rotationally symmetric and $([0,+\infty),dr^2)\times_{{F'(r)}^2}(\mathbb{S}^{n-1},{\bar g}_{S})$, or 
 \item[$(2)$] a product Riemannian manifold
$\left(
\mathbb{R}\times N^{n-1}, dr^2+\frac{\bar R}{\lambda}\bar g, \sqrt{\frac{\bar R}{\lambda}} r+\alpha
\right),
$
where $\bar R$ is a negative constant scalar curvature of $N$ and $\alpha$ is a constant.
\end{enumerate}
\end{theorem}

\begin{proof}
We consider (2) of Theorem \ref{CSZ}.
By Lemma \ref{equivlambda}, we only need to consider $R<\lambda$. By assumption, we have $\rho'<0$ on $\mathbb{R}$.
If $\rho''<0$ on $\mathbb{R}$, then $\rho$ is a positive concave function, which is a contradiction. 

Assume that there exists $\tilde r\in\mathbb{R}$ such that $\rho''(\tilde r)=0$.
By \eqref{key3}, 
\[
\rho'''=-\frac{\rho'(\rho'+\lambda)}{n-1}<0,
\] 
at $\tilde r$.
Hence, $\rho''<0$ on an interval $(\tilde r,\tilde r_1)$ for some $\tilde r_1$.
Iterating this argument, we see that $\rho'$ is weakly decreasing on $(\tilde r,+\infty)$.
Since $\rho'<0$, we have $\rho'<-\gamma^2$ for some constant $\gamma$ on $(\tilde r_2,+\infty)$ for some $\tilde r_2(>\tilde r)$. Since $\rho>0$, we have a contradiction.
Hence, $\rho''>0$ on $\mathbb{R}$. 
Set $s=-r$. Since $\dot\rho(s)=\frac{d}{ds}\rho(s)=-\rho'(r)>0$ and $\ddot\rho(s)>0$, $\rho\nearrow+\infty$ as $s\nearrow+\infty$. 
Thus, for a sufficiently large $\tilde s$, $-\lambda\rho^2+\bar R>0$ on $(\tilde s,+\infty)$.
By  \eqref{key2}, we have 
\[
2(n-1)\rho\ddot\rho+(n-1)(n-2)\dot\rho^2>\dot\rho\rho^2,
\]
on $(\tilde s,+\infty)$. 
Hence, we have the following differential inequality
\[
\frac{d}{ds}\left(
\rho^{\frac{n-2}{2}}\dot\rho
\right)
>\frac{1}{(n-1)(n+2)}\frac{d}{ds}\rho^{\frac{n+2}{2}},
\]
on $(\tilde s,+\infty)$.
Integrating both sides, we obtain
\[
\dot\rho>\frac{1}{(n-1)(n+2)}\rho^2+\tilde C\rho^{-\frac{n-2}{2}},
\]
on $(\tilde s,+\infty)$, where $\tilde C$ is a constant.
Since $\rho$ is sufficiently large on $(\tilde s,+\infty)$, we have 
\[
\rho^{-2}\dot\rho>\frac{1}{2(n-1)(n+2)},
\]
on $(\tilde s,+\infty)$. For $\tilde s_1>\tilde s$, by integrating both sides from $\tilde s_1$ to $s$, we obtain
\[
-(\rho^{-1}(s)-\rho^{-1}(\tilde s_1))>\frac{1}{2(n-1)(n+2)}(s-\tilde s_1).
\]
By taking $s\nearrow+\infty$, we have a contradiction.
\end{proof}

\begin{remark}
By the same argument, we can classify steady solitons:

Let $(M^n,g,F)$ be a nontrivial, complete, steady gradient Yamabe soliton with nonpositive scalar curvature. Then, $(M^n,g,F)$ is either 
\begin{enumerate}
\item[$(1)$]
 rotationally symmetric and $([0,+\infty),dr^2)\times_{{F'(r)}^2}(\mathbb{S}^{n-1},{\bar g}_{S})$, or 
 \item[$(2)$] a product Riemannian manifold
$\left(
\mathbb{R}\times N^{n-1}(0), dr^2+a^2\bar g, ar+b
\right),
$
for some $a(\neq0),b\in\mathbb{R}$, where $N^{n-1}(0)$ has zero scalar curvature.
\end{enumerate}

\end{remark}


\section{Case $R\geq\lambda$}\label{sectionR>l}
 
In this section, we consider the case $R\geq\lambda$. 

\begin{theorem}\label{R>l}
Let $(M^n,g,F,\lambda)$ be a nontrivial, complete, expanding gradient Yamabe soliton such that 
$M=\mathbb{R}\times N^{n-1}$ and $g=dr^2+{{F'(r)}^2} \bar g$. 
If $R\geq\lambda$, then $(M,g,F)$ is either
\begin{enumerate}
\item[$(1)$]
a product Riemannian manifold
$\left(
\mathbb{R}\times N^{n-1}, dr^2+\frac{\bar R}{\lambda}\bar g, \sqrt{\frac{\bar R}{\lambda}} r+\alpha
\right), 
$
where $\bar R$ is a negative constant scalar curvature of $N$ and $\alpha$ is a constant, or
\item[$(2)$]
a warped product manifold
\[
(\mathbb{R},dr^2)\times_{|\nabla F|}(N^{n-1},\bar g),
\]
where $N$ has constant scalar curvature $\bar R$. 
Furthermore, $\bar R\leq 0$ and $\lambda<R<0$. 
Moreover, the potential function $F$ satisfies $\displaystyle\lim_{r\rightarrow-\infty}F'(r)=\sqrt{\frac{\bar R}{\lambda}}$, $F''(r)>0$, and $F'''(r)\geq0$ on $\mathbb{R}$. However, there is no open interval on which $F'''(r)=0$. 
\end{enumerate}
\end{theorem}

As a corollary, we can classify expanding gradient Yamabe solitons with $R\geq\lambda.$
\begin{corollary}
A nontrivial, complete, expanding gradient Yamabe soliton $(M^n,g,F,\lambda)$ with $R\geq\lambda$ is one of the following:

\begin{enumerate}
\item[$(1)$]
a product Riemannian manifold
$\left(
\mathbb{R}\times N^{n-1}, dr^2+\frac{\bar R}{\lambda}\bar g, \sqrt{\frac{\bar R}{\lambda}} r+\alpha
\right), 
$
where $\bar R$ is a negative constant scalar curvature of $N$ and $\alpha$ is a constant, or
\item[$(2)$]
a manifold with rotational symmetry 
\[
([0,+\infty),dr^2)\times_{|\nabla F|}(\mathbb{S}^{n-1},\bar g_S),
\]
where $\bar g_S$ is the round metric on $\mathbb{S}^{n-1}$, or
\item[$(3)$]
a warped product manifold
\[
(\mathbb{R},dr^2)\times_{|\nabla F|}(N^{n-1}(0),\bar g),
\]
which satisfies $\lambda<R<0$, where $(N^{n-1}(0),\bar g)$ has zero scalar curvature, or
\item[$(4)$]
a warped product manifold
\[
(\mathbb{R},dr^2)\times_{|\nabla F|}(N^{n-1}(-c^2),\bar g),
\]
which satisfies $\lambda<R<0$, where $(N^{n-1}(-c^2),\bar g)$ has negative constant scalar curvature.
\end{enumerate}

Furthermore, in cases $(3)$ and $(4)$, the potential function $F$ satisfies $\displaystyle\lim_{r\rightarrow-\infty}F'(r)=\sqrt{\frac{\bar R}{\lambda}}$, $F''(r)>0$ and $F'''(r)\geq0$ on $\mathbb{R}$. However, there is no open interval on which $F'''(r)=0$. 
 \end{corollary}

\begin{proof}[Proof of Theorem $\ref{R>l}$]
By Lemma \ref{equivlambda}, we consider $R>\lambda$.
First, we show that the scalar curvature $\bar R$ of the manifold $N$ appearing in Case (2) of Theorem \ref{CSZ} is nonpositive.
\begin{claim}\label{claim1}
The scalar curvature $\bar R$ of $N$ in Case (2) of Theorem \ref{CSZ} is nonpositive.
\end{claim}
\begin{proof}[Proof of Claim \ref{claim1}]
Assume that $\bar R$ is positive. 
Set $s=-r$. Since $\dot\rho(s)=-\rho'(r)(<0)$, $\ddot\rho(s)=\rho''(r)$, and $\dddot\rho(s)=-\rho'''(r)$, the equations $\eqref{key2}$ and $\eqref{key3}$ are as follows.
\begin{equation}\label{eq12}
2(n-1)\rho\ddot\rho+(n-1)(n-2)\dot\rho^2+\rho^2(-\dot\rho+\lambda)=\bar R,
\end{equation}
and
\begin{equation*}
-2(n-1)^2\dot\rho\ddot\rho-2(n-1)\rho\dddot \rho-2\rho\dot\rho(-\dot\rho+\lambda)+\rho^2\ddot\rho=0.
\end{equation*}

If $\ddot\rho<0$ on $\mathbb{R}$, then we have a contradiction, because $\rho>0$.

Assume that there exists $\tilde s\in\mathbb{R}$ such that $\ddot\rho(\tilde s)=0$. 
Then one has 
\[
\dddot\rho(\tilde s)=-\frac{\dot\rho(\tilde s)(-\dot\rho(\tilde s)+\lambda)}{n-1}.
\]
We have three cases.

If $-\dot\rho(\tilde s)+\lambda>0$, then $\dddot\rho(\tilde s)>0$.
By \eqref{eq12}, we have
\[
\dot\rho=\frac{\rho^2\pm\sqrt{\rho^4-4(n-1)(n-2)(\lambda\rho^2-\bar R)}}{2(n-1)(n-2)},
\]
at $\tilde s$.
Since $\dot\rho<0$, one has
\begin{align*}
-2(n-1)(n-2)\lambda+\rho^2
<-2(n-1)(n-2)\dot\rho+\rho^2
=\sqrt{\rho^4-4(n-1)(n-2)(\lambda\rho^2-\bar R)},
\end{align*}
at $\tilde s$.
Since both sides of the above inequality are positive, one has 
\[
(n-1)(n-2)\lambda^2<\bar R.
\]

By the same argument, if $-\dot\rho(\tilde s)+\lambda<0$, that is, $\dddot\rho(\tilde s)<0$, then one has 
\[
(n-1)(n-2)\lambda^2>\bar R,
\]
and 
if $-\dot\rho(\tilde s)+\lambda=0$, that is, $\dddot\rho(\tilde s)=0$, then one has 
\[
(n-1)(n-2)\lambda^2=\bar R.
\]
Therefore, if there exists a point $\tilde s$ such that $\ddot\rho=0$, $\dddot\rho$ has the same behavior. 

Case 1. If $-\dot\rho+\lambda>0$ at the point $\tilde s$, then $\dddot\rho(\tilde s)>0$ and $\dot\rho$ is weakly increasing on $(\tilde s,+\infty)$. 
One has $\rho\searrow c(\geq0)$ and $\dot\rho\nearrow 0$ as $s\nearrow+\infty$. 

If $c>0$, then by \eqref{eq12}, 
\[
\ddot\rho\rightarrow\frac{\bar R-c^2\lambda}{2(n-1)c}(>0),
\]
as $s\nearrow+\infty$. Since $\dot\rho<0$, we have a contradiction.

If $c=0$, then by \eqref{eq12} again, $\ddot\rho\nearrow+\infty$ as $s\nearrow +\infty$.
Since $\dot\rho<0$, we have a contradiction.

Case 2. If $-\dot\rho+\lambda<0$ at the point $\tilde s$, then $\dddot\rho(\tilde s)<0$ and $\dot\rho$ is weakly decreasing on $(\tilde s,+\infty)$. Hence, one has $\dot\rho<-\gamma^2$ for some constant $\gamma$ on $(\tilde s_2,+\infty)$ for some $\tilde s_2$. Since $\rho>0$, we have a contradiction.

Case 3. If $-\dot\rho+\lambda=0$ at the point $\tilde s$, then $(n-1)(n-2)\lambda^2=\bar R$. By \eqref{eq12}, we have 
\begin{equation}\label{claimdd}
2(n-1)\rho\ddot\rho
=(-\dot\rho+\lambda)
\{
-\rho^2+(n-1)(n-2)(\lambda+\dot\rho)
\}.
\end{equation}
Hence, the points satisfying $\ddot{\rho}=0$ coincide with those satisfying $\dot{\rho}=\lambda$. Therefore, we only need to consider the following three cases.

Case 1. $\dot\rho>\lambda$ on $(\tilde s_3,+\infty)$ for some $\tilde s_3(>\tilde s)$. 
In this case, $\ddot\rho>0$ on $(\tilde s_3,+\infty)$. Hence, by \eqref{claimdd} again, an elementary argument shows that $\ddot\rho>\gamma(>0)$ for some constant $\gamma$ on $(\tilde s_4,+\infty)$ for sufficiently large $\tilde s_4$.
Since $\dot\rho<0$ on $\mathbb{R}$, we have a contradiction.

Case 2. $\dot\rho<\lambda$ on $(\tilde s_3,+\infty)$ for some $\tilde s_3$.
In this case, we have $\ddot\rho<0$ on $(\tilde s_3,+\infty)$.
Since $\rho>0$ on $\mathbb{R}$, we have a contradiction.

Case 3. $\dot\rho=\lambda$ on $(\tilde s_3,+\infty)$ for some $\tilde s_3$.
Since $\rho>0$ on $\mathbb{R}$, we have a contradiction.

Therefore, there is no such point $\tilde s$. Hence, one has $\ddot\rho>0$ on $\mathbb{R}$. 
We have $\rho\searrow c(\geq0)$ for some $c\in\mathbb{R}$ and $\dot\rho\nearrow0$ as $s\nearrow+\infty$.

Case 1: $c>0$. By \eqref{eq12}, we have
\[
\ddot\rho\rightarrow \frac{\bar R-\lambda c^2}{2(n-1)c}(>0),
\]
as $s\nearrow+\infty$. Since $\dot\rho<0$, we have a contradiction.

Case 2: $c=0$. By \eqref{eq12}, we have $\ddot\rho\rightarrow+\infty$ as $s\nearrow+\infty$. Since $\dot\rho<0$, we have a contradiction.
\end{proof}

Therefore, one has $\bar R\leq0.$
Under this assumption, the following claim holds.
\begin{claim}\label{claim2}
If there exists a point $r_0\in\mathbb{R}$ such that $\rho''(r_0)=0$, then $\rho'''(r_0)>0$.
Furthermore, we also have $\sqrt{\frac{\bar R}{\lambda}}<\rho(r_0).$ 
\end{claim}

\begin{proof}[Proof of Claim \ref{claim2}]~\\
\noindent
Case 1. $\bar R=0.$ 
Assume that $\rho''(r_0)=0$. By $\eqref{key2}$, 
\[
(n-1)(n-2)\rho'^2(r_0)+\rho^2(r_0)\rho'(r_0)=-\lambda\rho^2(r_0).
\]
Hence, one has
\[
\rho'(r_0)=\frac{-\rho^2(r_0)\pm\sqrt{\rho^4(r_0)-4(n-1)(n-2)\lambda\rho^2(r_0)}}{2(n-1)(n-2)}.
\]
Since $\rho'>0$, it must be
\begin{equation}\label{rp0}
\rho'(r_0)=\frac{-\rho^2(r_0)+\sqrt{\rho^4(r_0)-4(n-1)(n-2)\lambda\rho^2(r_0)}}{2(n-1)(n-2)}.
\end{equation}
Furthermore, by $\eqref{key3}$, 
\[
\rho'''(r_0)=-\frac{1}{n-1}\rho'(r_0)(\rho'(r_0)+\lambda).
\]

Assume that $-\lambda<\rho'(r_0)$. 
Combining this inequality with \eqref{rp0}, we have
\[
-\lambda
<
\frac{-\rho^2(r_0)+\sqrt{\rho^4(r_0)-4(n-1)(n-2)\lambda\rho^2(r_0))}}{2(n-1)(n-2)}.
\]
A straightforward computation leads to a contradiction.

Assume that $-\lambda=\rho'(r_0)$. 
A similar argument shows that 
$\lambda=0$, which is a contradiction.

Hence, the inequality $-\lambda>\rho'(r_0)$ holds and one has
$\rho'''(r_0)>0.$
\\

\noindent
Case 2. $\bar R<0$. 
Assume that $\rho''(r_0)=0$. By $\eqref{key2}$, 
\[
(n-1)(n-2)\rho'^2(r_0)+\rho^2(r_0)\rho'(r_0)=\bar R-\lambda\rho^2(r_0).
\]
Hence, one has
\[
\rho^4(r_0)-4\lambda(n-1)(n-2)\rho^2(r_0)+4(n-1)(n-2)\bar R>0,
\]
and
\[
\rho'(r_0)=\frac{-\rho^2(r_0)\pm\sqrt{\rho^4(r_0)+4(n-1)(n-2)(\bar R-\lambda\rho^2(r_0))}}{2(n-1)(n-2)}.
\]
Since $\rho'>0$, one has 
\[
\sqrt{\frac{\bar R}{\lambda}}<\rho(r_0),
\]
and
\[
\rho'(r_0)=\frac{-\rho^2(r_0)+\sqrt{\rho^4(r_0)+4(n-1)(n-2)(\bar R-\lambda\rho^2(r_0))}}{2(n-1)(n-2)}.
\]
Furthermore, by $\eqref{key3}$, 
\[
\rho'''(r_0)=-\frac{1}{n-1}\rho'(r_0)(\rho'(r_0)+\lambda).
\]
Since $\rho'>0$, a similar argument shows that $\rho'''(r_0)>0.$
\end{proof}

Assume that there exists $r_1\in\mathbb{R}$ such that $\rho'(r_1)=-\lambda.$ By $\eqref{key2}$, we have
$\rho''(r_1)<0.$ 
Hence, there exists no point $r_{-1}\in(-\infty,r_1)$ such that $\rho''(r_{-1})=0$.
In fact, if there exists such a point, then by Claim \ref{claim2}, $\rho'''(r_{-1})>0$ and the function $\rho'$ is weakly increasing on $(r_{-1},+\infty)$. For the same reason, it cannot occur that $\rho''>0$ 
at some point 
on $(-\infty,r_1)$. Therefore, we have
 $\rho''<0$ on $(-\infty,r_2)$ for some $r_2(>r_1)$. However, since $\rho>0,\rho'>0$, and $\rho''<0$ on $(-\infty, r_2)$, we have a contradiction. Hence, such a point $r_1$ does not exist, and one has either $\rho'>-\lambda$ or $\rho'<-\lambda$ on $\mathbb{R}$. By Proposition \ref{R>le}, the first case cannot occur.
Therefore, $\lambda<R<0.$ 

Claim \ref{claim2} shows that for each critical point $r_0$ of $\rho'$, the function $\rho'$ is weakly increasing on $(r_0,+\infty)$. Assume that $\rho'$ has a critical point $r_0$. Furthermore, if there exists an interval on which $\rho'$ is decreasing, then without loss of generality, we may assume that it is $(-\infty,r_0)$. However, by the same argument as before, we have a contradiction.
Therefore, $\rho'$ is weakly increasing on $\mathbb{R}$. 
If there exists no critical point of $\rho'$, we have $\rho''>0$ or $\rho''<0$ on $\mathbb{R}$. However, $\rho''<0$ cannot occur. Hence, one has $\rho''>0$.

Therefore, $\rho''\geq0$ on $\mathbb{R}$ and there exists no open interval on which $\rho''=0$. 
Set $s=-r$ as in the proof of Claim \ref{claim1}. 
Since $\rho>0$ and $\dot\rho<0$, we have $\rho\searrow c(\geq0)$ for some $c\in\mathbb{R}$, and $\dot\rho\nearrow0$ as $s\nearrow +\infty$. 

Case 1: $c>0$. By \eqref{eq12}, $\ddot\rho$ must converge to some constant $C(\geq0)$. If $C>0$, since $\dot\rho<0$, we have a contradiction. If $C=0$, we have $\lambda c^2=\bar R$.

Case 2: $c=0$. By \eqref{eq12} again, since $\ddot\rho\geq0$, $\bar R$ must be zero. 

Therefore, we have 
\[
\bar R=\lambda\lim_{r\rightarrow-\infty}\rho^2(r),
\]
and $F'(r)>\sqrt{\frac{\bar R}{\lambda}}$ on $\mathbb{R}$.
\end{proof}







\bibliographystyle{amsbook}

\end{document}